\documentclass{amsart}
\usepackage{amssymb}
\usepackage{amsmath}
\usepackage{amsfonts}

\newtheorem{theorem}{Theorem}
\theoremstyle{plain}

\newtheorem{definition}{Definition}
\newtheorem{example}{Example}

\newtheorem{proposition}{Proposition}
\newtheorem{remark}{Remark}

\numberwithin{equation}{section}
\input{tcilatex}

\begin{document}
\title[Some Fixed-Circle Results in $C^{\ast }-$Algebra Valued Metric Spaces]%
{Some Fixed-Circle Results in $C^{\ast }-$Algebra Valued Metric Spaces}
\author{Nilay DE\u{G}\.{I}RMEN}
\address[N. De\u{g}irmen]{Ondokuz Mayis University Faculty of Art and
Sciences Deparment of Mathematics Samsun, Turkey}
\email[N. De\u{g}irmen]{nilay.sager@omu.edu.tr}

\begin{abstract}
In this paper, we consider fixed-circle problem in $C^{\ast }-$algebra
valued metric spaces and prove some fixed-circle theorems for self-mappings
by defining the notion of fixed-circle on such spaces with geometric
interpretation. Furthermore, we give some illustrative examples to
substantiate the importance of our newly obtained results.
\end{abstract}

\maketitle

\section{Introduction and Preliminaries}

The Banach contraction principle \cite{1} is a popular and effective tool
used for the existence and uniqueness of solutions of many nonlinear
problems arising in physics and engineering sciences. Up to now, researchers
generalized the Banach contraction principle in many directions and obtained
new results in different metric spaces. With this idea in mind, in 2014, Ma 
\textit{et al.} \cite{2} introduced the concept of $C^{\ast }-$algebra
valued metric space and established some fixed point theorems for
self-mappings satisfying the contractive conditions on such spaces. Also, in
2015, Batul and Kamran \cite{3} generalized it by weakening the contractive
condition introduced by Ma \textit{et al}. \cite{2}. Going in the same
direction, many articles on fixed point results in $C^{\ast }-$algebra
valued metric spaces, we refer to \cite{4,5,6,7,8,9}.

On the other hand, as a geometric approach to the fixed point theory, in 
\cite{10}, \"{O}zg\"{u}r and Ta\c{s} initiated the investigations concerning
a fixed-circle problem in metric spaces and examined some fixed-circle
theorems for self-mappings on metric spaces with geometric interpretation by
giving some necessary examples to validate their own findings. Since this
subject has been developed very fast in recent times due to theoretical
mathematical studies and some applications in different fields of
mathematical sciences such as neural networks, it has attracted considerable
interest from many authors. New solutions of fixed-circle problem was
investigated with various aspects and new contractive conditions on both
metric spaces and some generalized metric spaces. For fixed-circle results
using different techniques, we recommend \cite{11,12,13,14,15,16}.

In the most existing literature, it has been observed that fixed circles
have not been defined in $C^{\ast }-$algebra valued metric spaces. This
motivates us to analyze solutions of the fixed-circle problem in such spaces
with geometric interpretation. So, our purpose in this study is to define
the notion of a fixed circle on a $C^{\ast }-$algebra valued metric space
and obtain some fixed-circle theorems for self-mappings on $C^{\ast }-$%
algebra valued metric spaces. Also, we construct some nontrivial
illustrative examples of mappings which have or not fixed circles. Since the
theory of $C^{\ast }-$algebra is one of the most extensive research areas in
operator theory and functional analysis, which is extremely active and
having huge applications to theoretical physics and noncommutative geometry,
we believe that the results of this article will contribute to studying on
fixed-circles with different aspects for $C^{\ast }-$algebra valued metric
spaces in the future.

Now, we summarize a number of known definitions and results about $C^{\ast
}- $algebras and $C^{\ast }-$algebra valued metric spaces , which will be
needed in our subsequent discussions.

We start with the definition of a $C^{\ast }-$algebra and its some
properties used in this research.

\begin{definition}
\cite{17} A mapping $x\rightarrow x^{\ast }$ of a complex algebra $\mathbb{A}
$ into $\mathbb{A}$ is called an involution on $\mathbb{A}$ if the following
properties hold for all $x,y\in \mathbb{A}$ and $\lambda \in \mathbb{C}:$
\end{definition}

\textit{(i) }$\left( x^{\ast }\right) ^{\ast }=x,$

\textit{(ii) }$\left( xy\right) ^{\ast }=y^{\ast }x^{\ast },$

\textit{(iii) }$\left( \lambda x+y\right) ^{\ast }=\overline{\lambda }%
x^{\ast }+y^{\ast }.$

\textit{A complex Banach algebra }$A$\textit{\ with an involution such that
for every }$x$\textit{\ in }$A$%
\begin{equation*}
\left\Vert x^{\ast }x\right\Vert =\left\Vert x\right\Vert ^{2}
\end{equation*}%
\textit{is called a }$C^{\ast }-$\textit{algebra.}

In the rest of paper, $\mathbb{A}$ will denote a unital $C^{\ast }-$algebra
with a unit $I.$

Let $\mathbb{A}_{h}=\left\{ x\in \mathbb{A}:x=x^{\ast }\right\} .$ An
element $x\in \mathbb{A}$ is called positive if $x\in \mathbb{A}_{h}$ and $%
\sigma \left( x\right) \subset 
%TCIMACRO{\U{211d} }%
%BeginExpansion
\mathbb{R}
%EndExpansion
^{+}$ where $\sigma \left( x\right) =\left\{ \lambda \in 
%TCIMACRO{\U{211d} }%
%BeginExpansion
\mathbb{R}
%EndExpansion
:x-\lambda I\text{ is non-invertible}\right\} ,$ the spectrum of $x$ in $%
\mathbb{A}$. If $x\in \mathbb{A}$ is positive, we write it as $\theta
\preceq x,$ where $\theta $ is the zero element in $\mathbb{A}.$ We denote
the set of all positive elements of $\mathbb{A}$ by $\mathbb{A}_{+}.$ Also, $%
\mathbb{A}_{h}$ becomes a partially ordered set by defining $x\preceq y$ to
mean $y-x\in \mathbb{A}_{+}$ \cite{17}$.$

The following statements about involution on $\mathbb{A}$, positive elements
of $\mathbb{A}$ and partial order $\preceq $ on $\mathbb{A}_{h}$ are true:

(i) $\left\Vert x^{\ast }\right\Vert =\left\Vert x\right\Vert $ for all $%
x\in \mathbb{A}$.

(ii) If $x\in \mathbb{A}$ is invertible, then $x^{\ast }$ is invertible and $%
\left( x^{\ast }\right) ^{-1}=\left( x^{-1}\right) ^{\ast }.$

(iii) If $x,y,z\in \mathbb{A}_{h}$, then $x\preceq y$ implies $x+z\preceq
y+z.$

(iv) If $x,y\in \mathbb{A}_{+}$ and $\alpha ,\beta \in 
%TCIMACRO{\U{211d} }%
%BeginExpansion
\mathbb{R}
%EndExpansion
^{+}\cup \left\{ 0\right\} ,$ then $\alpha x+\beta y\in $ $\mathbb{A}_{+}.$

(v) $\mathbb{A}_{+}=\left\{ x^{\ast }x:x\in \mathbb{A}\right\} $ .

(vi) If $x,y\in \mathbb{A}_{h}$ and $z\in \mathbb{A}$, then $x\preceq y$
implies $z^{\ast }xz\preceq z^{\ast }yz.$

(vii) If $\theta \preceq x\preceq y,$ then $\left\Vert x\right\Vert \leq
\left\Vert y\right\Vert $ \cite{17}$.$

Using the concept of a positive element in a $C^{\ast }-$algebra, in 2014,
Ma and Jiang \cite{2} introduced the notion of a $C^{\ast }-$algebra valued
metric space in the following way:

\begin{definition}
\cite{2} Let $X$ be a nonempty set. Suppose the mapping $d:X\times
X\rightarrow \mathbb{A}$ satisfies:
\end{definition}

\textit{(i) }$\theta \preceq d\left( x,y\right) $\textit{\ for all }$x,y\in
X $\textit{\ and }$d\left( x,y\right) =\theta \Leftrightarrow x=y;$

\textit{(ii) }$d\left( x,y\right) =d\left( y,x\right) $\textit{\ for all }$%
x,y\in X;$

\textit{(iii) }$d\left( x,y\right) \preceq d\left( x,z\right) +d\left(
z,y\right) $\textit{\ for all }$x,y,z\in X.$

\textit{Then, }$d$\textit{\ is called }$C^{\ast }-$\textit{algebra valued
metric on }$X$\textit{\ and }$\left( X,\mathbb{A},d\right) $\textit{\ is
called a }$C^{\ast }-$\textit{algebra valued metric space.}

It is clear that such spaces generalize the concept of metric spaces. The
main idea consists in using the set of all positive elements of a unital $%
C^{\ast }-$algebra instead of the set of real numbers.

\begin{example}
\cite{2} Let $E$ be a Lebesgue measurable set and $L\left( L^{2}\left(
E\right) \right) $ denote the set of bounded linear operators on Hilbert
space $L^{2}\left( E\right) .$ Define $d:L^{\infty }\left( E\right) \times
L^{\infty }\left( E\right) \rightarrow L\left( L^{2}\left( E\right) \right) $
by%
\begin{equation*}
d\left( f,g\right) =\pi _{\left\vert f-g\right\vert }
\end{equation*}%
for all $f,g\in L^{\infty }\left( E\right) $, where $\pi _{h}:L^{2}\left(
E\right) \rightarrow L^{2}\left( E\right) $ is the multiplication operator
defined by%
\begin{equation*}
\pi _{h}\left( \varphi \right) =h\cdot \varphi
\end{equation*}%
for all $\varphi \in L^{2}\left( E\right) .$ Then, $d$ is a $C^{\ast }-$%
algebra valued metric and $\left( L^{\infty }\left( E\right) ,L\left(
L^{2}\left( E\right) \right) ,d\right) $ is a complete $C^{\ast }-$algebra
valued metric space.
\end{example}

\begin{example}
\cite{2} Define $d:%
%TCIMACRO{\U{211d} }%
%BeginExpansion
\mathbb{R}
%EndExpansion
\times 
%TCIMACRO{\U{211d} }%
%BeginExpansion
\mathbb{R}
%EndExpansion
\rightarrow M_{2}\left( 
%TCIMACRO{\U{211d} }%
%BeginExpansion
\mathbb{R}
%EndExpansion
\right) $ by%
\begin{equation*}
d\left( x,y\right) =diag\left( \left\vert x-y\right\vert ,\alpha \left\vert
x-y\right\vert \right)
\end{equation*}%
for all $x,y\in 
%TCIMACRO{\U{211d} }%
%BeginExpansion
\mathbb{R}
%EndExpansion
$, where $\alpha \geq 0$ is a constant, the norm on $M_{2}\left( 
%TCIMACRO{\U{211d} }%
%BeginExpansion
\mathbb{R}
%EndExpansion
\right) $ is defined by%
\begin{equation*}
\left\Vert \left[ 
\begin{array}{cc}
a & b \\ 
c & d%
\end{array}%
\right] \right\Vert =\max \left\{ \left\vert a\right\vert ,\left\vert
b\right\vert ,\left\vert c\right\vert ,\left\vert d\right\vert \right\} ,
\end{equation*}%
and partial ordering on $M_{2}\left( 
%TCIMACRO{\U{211d} }%
%BeginExpansion
\mathbb{R}
%EndExpansion
\right) $ is given by%
\begin{equation*}
\left[ 
\begin{array}{cc}
a_{1} & b_{1} \\ 
c_{1} & d_{1}%
\end{array}%
\right] \preceq \left[ 
\begin{array}{cc}
a_{2} & b_{2} \\ 
c_{2} & d_{2}%
\end{array}%
\right] \Longleftrightarrow a_{1}\leq a_{2},\text{ }b_{1}\leq b_{2},\text{ }%
c_{1}\leq c_{2}\text{ and }d_{1}\leq d_{2}.
\end{equation*}%
Then, $d$ is a $C^{\ast }-$algebra valued metric and $\left( 
%TCIMACRO{\U{211d} }%
%BeginExpansion
\mathbb{R}
%EndExpansion
,M_{2}\left( 
%TCIMACRO{\U{211d} }%
%BeginExpansion
\mathbb{R}
%EndExpansion
\right) ,d\right) $ is a $C^{\ast }-$algebra valued metric space.
\end{example}

\begin{example}
\cite{3} Let $X=\left[ -1,1\right] \times \left[ -1,1\right] .$ Define $%
d:X\times X\rightarrow 
%TCIMACRO{\U{211d} }%
%BeginExpansion
\mathbb{R}
%EndExpansion
^{2}$ by%
\begin{equation*}
d\left( x,y\right) =\left( \left\vert x_{1}-y_{1}\right\vert ,\left\vert
x_{2}-y_{2}\right\vert \right)
\end{equation*}%
for all $x=\left( x_{1},x_{2}\right) ,y=\left( y_{1},y_{2}\right) \in X,$
where the norm on $%
%TCIMACRO{\U{211d} }%
%BeginExpansion
\mathbb{R}
%EndExpansion
^{2}$ is defined by%
\begin{equation*}
\left\Vert \left( a,b\right) \right\Vert =\max \left\{ \left\vert
a\right\vert ,\left\vert b\right\vert \right\} ,
\end{equation*}%
and partial ordering on $%
%TCIMACRO{\U{211d} }%
%BeginExpansion
\mathbb{R}
%EndExpansion
^{2}$ is given by%
\begin{equation*}
\left( a,b\right) \preceq \left( c,d\right) \Longleftrightarrow a\leq c\text{
and }b\leq d.
\end{equation*}%
Then, $d$ is a $C^{\ast }-$algebra valued metric and $\left( X,%
%TCIMACRO{\U{211d} }%
%BeginExpansion
\mathbb{R}
%EndExpansion
^{2},d\right) $ is a $C^{\ast }-$algebra valued metric space.
\end{example}

\begin{example}
\cite{6} Define $d:%
%TCIMACRO{\U{211d} }%
%BeginExpansion
\mathbb{R}
%EndExpansion
\times 
%TCIMACRO{\U{211d} }%
%BeginExpansion
\mathbb{R}
%EndExpansion
\rightarrow M_{2}\left( 
%TCIMACRO{\U{211d} }%
%BeginExpansion
\mathbb{R}
%EndExpansion
\right) $ by%
\begin{equation*}
d\left( x,y\right) =\left\{ \left[ 
\begin{array}{c}
\left[ 
\begin{array}{cc}
1 & 0 \\ 
0 & 1%
\end{array}%
\right] ,\text{ \ }x\neq y \\ 
0,\text{ \ \ \ \ \ \ \ \ \ }x=y%
\end{array}%
\right] \right.
\end{equation*}%
for all $x,y\in 
%TCIMACRO{\U{211d} }%
%BeginExpansion
\mathbb{R}
%EndExpansion
$. Then, $d$ is a $C^{\ast }-$algebra valued metric and $\left( 
%TCIMACRO{\U{211d} }%
%BeginExpansion
\mathbb{R}
%EndExpansion
,M_{2}\left( 
%TCIMACRO{\U{211d} }%
%BeginExpansion
\mathbb{R}
%EndExpansion
\right) ,d\right) $ is a $C^{\ast }-$algebra valued metric space.
\end{example}

Based on the idea of the Banach contraction principle \cite{1} in classical
metric spaces, Ma and Jiang \cite{2} established the following main theorems
which implies the existence and uniqueness of fixed point on complete $%
C^{\ast }-$algebra valued metric spaces.

\begin{theorem}
\cite{2} If $\left( X,\mathbb{A},d\right) $ is a complete $C^{\ast }-$%
algebra valued metric space and $T:X\rightarrow X$ is a $C^{\ast }-$algebra
valued contractive mapping on $X$, that is, there exists an $A\in \mathbb{A}$
with $\left\Vert A\right\Vert <1$ such that%
\begin{equation}
d\left( Tx,Ty\right) \preceq A^{\ast }d\left( x,y\right) A
\end{equation}%
for all $x,y\in X,$ then, there exists a unique fixed point in $X.$
\end{theorem}

In the following theorem and Theorem 12, we denote the set $\left\{ a\in 
\mathbb{A}:ab=ba\text{ for all }b\in \mathbb{A}\right\} $ by $\mathbb{A}%
^{\prime }.$

\begin{theorem}
\cite{2} Let $\left( X,\mathbb{A},d\right) $ be a complete $C^{\ast }-$%
algebra valued metric space. Suppose the mapping $T:X\rightarrow X$
satisfies for all $x,y\in X$ such that%
\begin{equation}
d\left( Tx,Ty\right) \preceq A\left( d\left( Tx,y\right) +d\left(
Ty,x\right) \right)
\end{equation}%
where $A\in \mathbb{A}_{+}^{\prime }$ and $\left\Vert A\right\Vert <\frac{1}{%
2}.$ Then, there exists a unique fixed point in $X.$
\end{theorem}

Inspired by preceding observation, many authors stated new types of
contractive mappings and studied fixed point theorems. Shehwar et al. \cite%
{4} gave the extension of Caristi's fixed point theorem \cite{18,19} for
self-mappings defined on $C^{\ast }-$algebra valued metric spaces, which
guarantees the existence of fixed point as follows:

\begin{theorem}
\cite{4} Let $\left( X,\mathbb{A},d\right) $ be a complete $C^{\ast }-$%
algebra valued metric space, $\phi :X\rightarrow \mathbb{A}_{+}$ be a lower
semi continuous map and $T:X\rightarrow X$ be such that%
\begin{equation}
d\left( x,Tx\right) \preceq \phi \left( x\right) -\phi \left( Tx\right)
\end{equation}%
for all $x\in X.$ Then, $T$ has at least one fixed point in $X.$
\end{theorem}

Kadelburg and Radenovic \cite{5} introduced the extension of \'{C}iri\'{c}'s
fixed point theorem \cite{20} for self-mappings defined on $C^{\ast }-$%
algebra valued metric spaces, which guarantees the existence and uniqueness
of fixed point as follows:

\begin{theorem}
\cite{5} Let $\left( X,\mathbb{A},d\right) $ be a $C^{\ast }-$algebra valued
metric space and $T:X\rightarrow X$ be a mapping on $X$. Suppose that there
exists an $A\in \mathbb{A}$ with $\left\Vert A\right\Vert <1$ such that for
all $x,y\in X$ there exists $u\left( x,y\right) \in \left\{ d\left(
x,y\right) ,d\left( x,Tx\right) ,d\left( y,Ty\right) ,d\left( x,Ty\right)
,d\left( y,Tx\right) \right\} $ such that 
\begin{equation}
d\left( Tx,Ty\right) \preceq A^{\ast }u\left( x,y\right) A.
\end{equation}%
Then, $T$ has a unique fixed point in $X.$
\end{theorem}

\section{Main Results}

We begin this section by introducing the concept of a circle on a $C^{\ast
}- $algebra valued metric space.

\begin{definition}
Let $\left( X,\mathbb{A},d\right) $ be a $C^{\ast }-$algebra valued metric
space, $x_{0}\in X$ and $r\in \mathbb{A}_{+}.$ Then, the circle with the
centered $x_{0}$ and the radius $r$ is defined by%
\begin{equation*}
C_{x_{0},r}^{C^{\ast }}=\left\{ x\in X:d\left( x,x_{0}\right) =r\right\} .
\end{equation*}
\end{definition}

\begin{example}
Consider the $C^{\ast }-$algebra valued metric space $\left( L^{\infty
}\left( E\right) ,L\left( L^{2}\left( E\right) \right) ,d\right) $ given in
Example 1 for $E=\left[ 0,1\right] $. Choose the center $x_{0}$ as the
function $f$ defined by%
\begin{equation*}
f:\left[ 0,1\right] \rightarrow 
%TCIMACRO{\U{211d} }%
%BeginExpansion
\mathbb{R}
%EndExpansion
,~~f\left( x\right) =\chi _{\left[ \frac{1}{2},1\right] }\left( x\right)
=\left\{ 
\begin{array}{c}
1,\text{ \ }x\in \left[ \frac{1}{2},1\right] \\ 
0,\text{ \ }x\notin \left[ \frac{1}{2},1\right]%
\end{array}%
\right.
\end{equation*}%
and the radius $r$ as the multiplication operator $\pi _{h}$ defined by%
\begin{equation*}
\pi _{h}:L^{2}\left[ 0,1\right] \rightarrow L^{2}\left[ 0,1\right] ,\text{ \ 
}\pi _{h}\left( \varphi \right) =h\cdot \varphi
\end{equation*}%
for the function $h:\left[ 0,1\right] \rightarrow 
%TCIMACRO{\U{211d} }%
%BeginExpansion
\mathbb{R}
%EndExpansion
,$\ $h\left( x\right) =\left\{ 
\begin{array}{c}
1,\text{ \ \ }x\in \left( 
%TCIMACRO{\U{211d} }%
%BeginExpansion
\mathbb{R}
%EndExpansion
\backslash 
%TCIMACRO{\U{211a} }%
%BeginExpansion
\mathbb{Q}
%EndExpansion
\right) \cap \left[ 0,1\right] \\ 
\infty ,\text{ \ \ }x\in 
%TCIMACRO{\U{211a} }%
%BeginExpansion
\mathbb{Q}
%EndExpansion
\cap \left[ 0,1\right] \text{ \ \ \ \ \ \ }%
\end{array}%
\right. $ . Then, it can be easily seen that $f,h\in L^{\infty }\left[ 0,1%
\right] $ and $\pi _{h}\in L\left( L^{2}\left[ 0,1\right] \right) .$ Thus,
we get%
\begin{eqnarray*}
C_{f,\pi _{h}}^{C^{\ast }} &=&\left\{ g\in L^{\infty }\left[ 0,1\right]
:d\left( g,f\right) =\pi _{h}\right\} \\
&=&\left\{ g\in L^{\infty }\left[ 0,1\right] :\pi _{\left\vert
g-f\right\vert }=\pi _{h}\right\} \\
&=&\left\{ g\in L^{\infty }\left[ 0,1\right] :\left\vert g-f\right\vert
=h\right\} \\
&=&\left\{ g\in L^{\infty }\left[ 0,1\right] :\left\vert g\left( x\right)
-f\left( x\right) \right\vert =h\left( x\right) \text{ for each }x\in \left[
0,1\right] \right\} \\
&=&\left\{ g\in L^{\infty }\left[ 0,1\right] :\left\vert g\left( x\right)
-0\right\vert =1\text{ for each }x\in \left( 
%TCIMACRO{\U{211d} }%
%BeginExpansion
\mathbb{R}
%EndExpansion
\backslash 
%TCIMACRO{\U{211a} }%
%BeginExpansion
\mathbb{Q}
%EndExpansion
\right) \cap \left[ 0,\frac{1}{2}\right) \right\} \\
&&\cup \left\{ g\in L^{\infty }\left[ 0,1\right] :\left\vert g\left(
x\right) -0\right\vert =\infty \text{ for each }x\in 
%TCIMACRO{\U{211a} }%
%BeginExpansion
\mathbb{Q}
%EndExpansion
\cap \left[ 0,\frac{1}{2}\right) \right\} \\
&&\cup \left\{ g\in L^{\infty }\left[ 0,1\right] :\left\vert g\left(
x\right) -1\right\vert =1\text{ for each }x\in \left( 
%TCIMACRO{\U{211d} }%
%BeginExpansion
\mathbb{R}
%EndExpansion
\backslash 
%TCIMACRO{\U{211a} }%
%BeginExpansion
\mathbb{Q}
%EndExpansion
\right) \cap \left[ \frac{1}{2},1\right] \right\} \\
&&\cup \left\{ g\in L^{\infty }\left[ 0,1\right] :\left\vert g\left(
x\right) -1\right\vert =\infty \text{ for each }x\in 
%TCIMACRO{\U{211a} }%
%BeginExpansion
\mathbb{Q}
%EndExpansion
\cap \left[ \frac{1}{2},1\right] \right\} \\
&=&\left\{ g\in L^{\infty }\left[ 0,1\right] :g\left( x\right) =-1\text{ or }%
g\left( x\right) =1\text{ for each }x\in \left( 
%TCIMACRO{\U{211d} }%
%BeginExpansion
\mathbb{R}
%EndExpansion
\backslash 
%TCIMACRO{\U{211a} }%
%BeginExpansion
\mathbb{Q}
%EndExpansion
\right) \cap \left[ 0,\frac{1}{2}\right) \right\} \\
&&\cup \left\{ g\in L^{\infty }\left[ 0,1\right] :g\left( x\right) =-\infty 
\text{ or }g\left( x\right) =\infty \text{ for each }x\in 
%TCIMACRO{\U{211a} }%
%BeginExpansion
\mathbb{Q}
%EndExpansion
\cap \left[ 0,\frac{1}{2}\right) \right\} \\
&&\cup \left\{ g\in L^{\infty }\left[ 0,1\right] :g\left( x\right) =0\text{
or }g\left( x\right) =2\text{ for each }x\in \left( 
%TCIMACRO{\U{211d} }%
%BeginExpansion
\mathbb{R}
%EndExpansion
\backslash 
%TCIMACRO{\U{211a} }%
%BeginExpansion
\mathbb{Q}
%EndExpansion
\right) \cap \left[ \frac{1}{2},1\right] \right\} \\
&&\cup \left\{ g\in L^{\infty }\left[ 0,1\right] :g\left( x\right) =-\infty 
\text{ or }g\left( x\right) =\infty \text{ for each }x\in 
%TCIMACRO{\U{211a} }%
%BeginExpansion
\mathbb{Q}
%EndExpansion
\cap \left[ \frac{1}{2},1\right] \right\} .
\end{eqnarray*}%
For example, the function $g$ defined by%
\begin{equation*}
g:\left[ 0,1\right] \rightarrow 
%TCIMACRO{\U{211d} }%
%BeginExpansion
\mathbb{R}
%EndExpansion
,~~g\left( x\right) =\left\{ 
\begin{array}{cc}
1 & \text{ \ }x\in \left( 
%TCIMACRO{\U{211d} }%
%BeginExpansion
\mathbb{R}
%EndExpansion
\backslash 
%TCIMACRO{\U{211a} }%
%BeginExpansion
\mathbb{Q}
%EndExpansion
\right) \cap \left[ 0,\frac{1}{7}\right) \\ 
-1 & \text{\ \ }x\in \left( 
%TCIMACRO{\U{211d} }%
%BeginExpansion
\mathbb{R}
%EndExpansion
\backslash 
%TCIMACRO{\U{211a} }%
%BeginExpansion
\mathbb{Q}
%EndExpansion
\right) \cap \left( \frac{1}{7},\frac{1}{2}\right) \\ 
\infty & x\in 
%TCIMACRO{\U{211a} }%
%BeginExpansion
\mathbb{Q}
%EndExpansion
\cap \left[ 0,\frac{1}{2}\right) \text{ }\ \ \text{\ } \\ 
2 & \text{ \ \ \ }x\in \left( 
%TCIMACRO{\U{211d} }%
%BeginExpansion
\mathbb{R}
%EndExpansion
\backslash 
%TCIMACRO{\U{211a} }%
%BeginExpansion
\mathbb{Q}
%EndExpansion
\right) \cap \left[ \frac{1}{2},\frac{99}{100}\right) \\ 
0 & \text{ \ \ }x\in \left( 
%TCIMACRO{\U{211d} }%
%BeginExpansion
\mathbb{R}
%EndExpansion
\backslash 
%TCIMACRO{\U{211a} }%
%BeginExpansion
\mathbb{Q}
%EndExpansion
\right) \cap \left[ \frac{99}{100},1\right] \\ 
-\infty & x\in 
%TCIMACRO{\U{211a} }%
%BeginExpansion
\mathbb{Q}
%EndExpansion
\cap \left[ \frac{1}{2},1\right] \text{ \ \ }%
\end{array}%
\right.
\end{equation*}%
is in the circle $C_{f,\pi _{h}}^{C^{\ast }}.$
\end{example}

\begin{example}
Consider the $C^{\ast }-$algebra valued metric space $\left( 
%TCIMACRO{\U{211d} }%
%BeginExpansion
\mathbb{R}
%EndExpansion
,M_{2}\left( 
%TCIMACRO{\U{211d} }%
%BeginExpansion
\mathbb{R}
%EndExpansion
\right) ,d\right) $ given in Example 2 for $\alpha =3$. Choose the center $%
x_{0}=0$ and the radii $r_{1}=\left[ 
\begin{array}{cc}
2 & 0 \\ 
0 & 6%
\end{array}%
\right] $ , $r_{2}=\left[ 
\begin{array}{cc}
2 & 0 \\ 
-1 & 5%
\end{array}%
\right] .$ Then, we get%
\begin{eqnarray*}
C_{0,r_{1}}^{C^{\ast }} &=&\left\{ x\in 
%TCIMACRO{\U{211d} }%
%BeginExpansion
\mathbb{R}
%EndExpansion
:d\left( x,0\right) =\left[ 
\begin{array}{cc}
2 & 0 \\ 
0 & 6%
\end{array}%
\right] \right\} \\
&=&\left\{ x\in 
%TCIMACRO{\U{211d} }%
%BeginExpansion
\mathbb{R}
%EndExpansion
:\left[ 
\begin{array}{cc}
\left\vert x\right\vert & 0 \\ 
0 & 3\left\vert x\right\vert%
\end{array}%
\right] =\left[ 
\begin{array}{cc}
2 & 0 \\ 
0 & 6%
\end{array}%
\right] \right\} \\
&=&\left\{ x\in 
%TCIMACRO{\U{211d} }%
%BeginExpansion
\mathbb{R}
%EndExpansion
:\left\vert x\right\vert =2\right\} \\
&=&\left\{ -2,2\right\}
\end{eqnarray*}%
and%
\begin{eqnarray*}
C_{0,r_{2}}^{C^{\ast }} &=&\left\{ x\in 
%TCIMACRO{\U{211d} }%
%BeginExpansion
\mathbb{R}
%EndExpansion
:d\left( x,0\right) =\left[ 
\begin{array}{cc}
2 & 0 \\ 
-1 & 5%
\end{array}%
\right] \right\} \\
&=&\left\{ x\in 
%TCIMACRO{\U{211d} }%
%BeginExpansion
\mathbb{R}
%EndExpansion
:\left[ 
\begin{array}{cc}
\left\vert x\right\vert & 0 \\ 
0 & 3\left\vert x\right\vert%
\end{array}%
\right] =\left[ 
\begin{array}{cc}
2 & 0 \\ 
-1 & 5%
\end{array}%
\right] \right\} \\
&=&\varnothing .
\end{eqnarray*}
\end{example}

\begin{example}
Consider the $C^{\ast }-$algebra valued metric space $\left( X,%
%TCIMACRO{\U{211d} }%
%BeginExpansion
\mathbb{R}
%EndExpansion
^{2},d\right) $ given in Example 3. Choose the center $x_{0}=\left( \frac{-1%
}{2},0\right) $ and the radius $r=\left( 1,1\right) .$ Then, we get%
\begin{eqnarray*}
C_{\left( \frac{-1}{2},0\right) ,\left( 1,1\right) }^{C^{\ast }} &=&\left\{
x=\left( x_{1},x_{2}\right) \in X:d\left( x,\left( \frac{-1}{2},0\right)
\right) =\left( 1,1\right) \right\} \\
&=&\left\{ x=\left( x_{1},x_{2}\right) \in X:\left\vert x_{1}+\frac{1}{2}%
\right\vert =\left\vert x_{2}\right\vert =1\right\} \\
&=&\left\{ \left( \frac{-3}{2},-1\right) ,\left( \frac{-3}{2},1\right)
,\left( \frac{1}{2},-1\right) ,\left( \frac{1}{2},1\right) \right\} .
\end{eqnarray*}
\end{example}

\begin{example}
Consider the $C^{\ast }-$algebra valued metric space $\left( 
%TCIMACRO{\U{211d} }%
%BeginExpansion
\mathbb{R}
%EndExpansion
,M_{2}\left( 
%TCIMACRO{\U{211d} }%
%BeginExpansion
\mathbb{R}
%EndExpansion
\right) ,d\right) $ given in Example 4. Choose the center $x_{0}=3$ and the
radius $r=\left[ 
\begin{array}{cc}
1 & 0 \\ 
0 & 1%
\end{array}%
\right] .$ Then, we get%
\begin{eqnarray*}
C_{3,r}^{C^{\ast }} &=&\left\{ x\in 
%TCIMACRO{\U{211d} }%
%BeginExpansion
\mathbb{R}
%EndExpansion
:d\left( x,3\right) =\left[ 
\begin{array}{cc}
1 & 0 \\ 
0 & 1%
\end{array}%
\right] \right\} \\
&=&\left\{ x\in 
%TCIMACRO{\U{211d} }%
%BeginExpansion
\mathbb{R}
%EndExpansion
:x\neq 3\right\} \\
&=&%
%TCIMACRO{\U{211d} }%
%BeginExpansion
\mathbb{R}
%EndExpansion
-\left\{ 3\right\} .
\end{eqnarray*}
\end{example}

\begin{example}
Define $d:%
%TCIMACRO{\U{211d} }%
%BeginExpansion
\mathbb{R}
%EndExpansion
\times 
%TCIMACRO{\U{211d} }%
%BeginExpansion
\mathbb{R}
%EndExpansion
\rightarrow M_{2}\left( 
%TCIMACRO{\U{211d} }%
%BeginExpansion
\mathbb{R}
%EndExpansion
\right) $ by%
\begin{equation*}
d\left( x,y\right) =diag\left( \left\vert e^{x}-e^{y}\right\vert ,\alpha
\left\vert e^{x}-e^{y}\right\vert \right)
\end{equation*}%
for all $x,y\in 
%TCIMACRO{\U{211d} }%
%BeginExpansion
\mathbb{R}
%EndExpansion
$, where $\alpha \geq 0$ is a constant. It is easy to verify $d$ is a $%
C^{\ast }-$algebra valued metric and $\left( 
%TCIMACRO{\U{211d} }%
%BeginExpansion
\mathbb{R}
%EndExpansion
,M_{2}\left( 
%TCIMACRO{\U{211d} }%
%BeginExpansion
\mathbb{R}
%EndExpansion
\right) ,d\right) $ is a complete $C^{\ast }-$algebra valued metric space.
If we choose the center $x_{0}=1$ and the radii $r_{1}=\left[ 
\begin{array}{cc}
2e & 0 \\ 
0 & 4e%
\end{array}%
\right] ,$ $r_{2}=\left[ 
\begin{array}{cc}
\frac{e}{2} & 0 \\ 
0 & e%
\end{array}%
\right] $ we get for $\alpha =2$%
\begin{eqnarray*}
C_{1,r_{1}}^{C^{\ast }} &=&\left\{ x\in 
%TCIMACRO{\U{211d} }%
%BeginExpansion
\mathbb{R}
%EndExpansion
:d\left( x,1\right) =\left[ 
\begin{array}{cc}
2e & 0 \\ 
0 & 4e%
\end{array}%
\right] \right\} \\
&=&\left\{ x\in 
%TCIMACRO{\U{211d} }%
%BeginExpansion
\mathbb{R}
%EndExpansion
:\left[ 
\begin{array}{cc}
\left\vert e^{x}-e\right\vert & 0 \\ 
0 & 2\left\vert e^{x}-e\right\vert%
\end{array}%
\right] =\left[ 
\begin{array}{cc}
2e & 0 \\ 
0 & 4e%
\end{array}%
\right] \right\} \\
&=&\left\{ x\in 
%TCIMACRO{\U{211d} }%
%BeginExpansion
\mathbb{R}
%EndExpansion
:\left\vert e^{x}-e\right\vert =2e\right\} \\
&=&\left\{ 1+\ln 3\right\}
\end{eqnarray*}%
and 
\begin{eqnarray*}
C_{1,r_{2}}^{C^{\ast }} &=&\left\{ x\in 
%TCIMACRO{\U{211d} }%
%BeginExpansion
\mathbb{R}
%EndExpansion
:d\left( x,1\right) =\left[ 
\begin{array}{cc}
\frac{e}{2} & 0 \\ 
0 & e%
\end{array}%
\right] \right\} \\
&=&\left\{ x\in 
%TCIMACRO{\U{211d} }%
%BeginExpansion
\mathbb{R}
%EndExpansion
:\left[ 
\begin{array}{cc}
\left\vert e^{x}-e\right\vert & 0 \\ 
0 & 2\left\vert e^{x}-e\right\vert%
\end{array}%
\right] =\left[ 
\begin{array}{cc}
\frac{e}{2} & 0 \\ 
0 & e%
\end{array}%
\right] \right\} \\
&=&\left\{ x\in 
%TCIMACRO{\U{211d} }%
%BeginExpansion
\mathbb{R}
%EndExpansion
:\left\vert e^{x}-e\right\vert =\frac{e}{2}\right\} \\
&=&\left\{ 1+\ln 3-\ln 2,1-\ln 2\right\} .
\end{eqnarray*}
\end{example}

\subsection{The existence of fixed circles}

In this part, we introduce the notion of a fixed circle on a $C^{\ast }-$%
algebra valued metric space. Then, we prove some fixed-circle theorems that
guarantee the existence of fixed-circles for a self-mapping satisfying some
conditions on $C^{\ast }-$algebra valued metric spaces.

Now, we define the new concept as follows:

\begin{definition}
Let $\left( X,\mathbb{A},d\right) $ be a $C^{\ast }-$algebra valued metric
space, $C_{x_{0},r}^{C^{\ast }}$ be a circle on $X$ and $T:X\rightarrow X$
be a self-mapping. If $Tx=x$ for all $x\in C_{x_{0},r}^{C^{\ast }},$ then
the circle $C_{x_{0},r}^{C^{\ast }}$ is called as the fixed circle of $T.$
\end{definition}

The first solution of fixed-circle problem in $C^{\ast }-$algebra valued
metric spaces is given using the inequality (1.3) in Theorem 3 which is an
extension of Caristi's fixed point theorem for mappings defined on $C^{\ast
}-$algebra valued metric spaces as follows:

\begin{theorem}
Let $\left( X,\mathbb{A},d\right) $ be a $C^{\ast }-$algebra valued metric
space, $C_{x_{0},r}^{C^{\ast }}$ be any circle on $X.$ Define the mapping $%
\varphi :X\rightarrow \mathbb{A}_{+}$ as%
\begin{equation}
\varphi \left( x\right) =d\left( x,x_{0}\right)
\end{equation}
for all $x\in X.$ If $T$ is a self-mapping defined on $X$ satisfying the
conditions%
\begin{equation}
d\left( x,Tx\right) \preceq \varphi \left( x\right) -\varphi \left( Tx\right)
\end{equation}%
and%
\begin{equation}
r\preceq d\left( Tx,x_{0}\right)
\end{equation}%
for all $x\in C_{x_{0},r}^{C^{\ast }},$ then the circle $C_{x_{0},r}^{C^{%
\ast }}$ is a fixed circle of $T.$
\end{theorem}

\begin{proof}
Let $x$ be any point in the circle $C_{x_{0},r}^{C^{\ast }}.$ Then, using
the (2.2), (2.1), (2.3) and the definition of the relation $\preceq ,$ we
get 
\begin{eqnarray*}
d\left( x,Tx\right) &\preceq &\varphi \left( x\right) -\varphi \left(
Tx\right) \\
&=&d\left( x,x_{0}\right) -d\left( Tx,x_{0}\right) \\
&=&r-d\left( Tx,x_{0}\right) \\
&\preceq &r-r=0
\end{eqnarray*}%
and so $d\left( x,Tx\right) =0$ which means that $Tx=x.$ Then, we obtain
that $C_{x_{0},r}^{C^{\ast }}$ is a fixed circle of $T.$
\end{proof}

\begin{remark}
Theorem 5 guarantees the existence of at least one fixed circle of a
self-mapping on a $C^{\ast }-$algebra valued metric space, while Theorem 3
guarantees the existence of at least one fixed point of a self-mapping on a $%
C^{\ast }-$algebra valued metric space. We note that if the circle $%
C_{x_{0},r}^{C^{\ast }}$ has only one element, Theorem 5 is a special case
of Theorem 3.
\end{remark}

\begin{remark}
The inequality (2.2) says that $Tx$ is not in the exterior of the circle $%
C_{x_{0},r}^{C^{\ast }}$ for each $x\in C_{x_{0},r}^{C^{\ast }}.$ In the
same way, the inequality (2.3) says that $Tx$ is not in the interior of the
circle $C_{x_{0},r}^{C^{\ast }}$ for each $x\in C_{x_{0},r}^{C^{\ast }}.$ It
follows that $T\left( C_{x_{0},r}^{C^{\ast }}\right) \subset
C_{x_{0},r}^{C^{\ast }}$ under the conditions (2.2) and (2.3).
\end{remark}

The following example illustrates Theorem 5.

\begin{example}
Consider the $C^{\ast }-$algebra valued metric space $\left( L^{\infty
}\left( E\right) ,L\left( L^{2}\left( E\right) \right) ,d\right) $ given in
Example 1 for $E=\left[ 0,1\right] $ and the circle $C_{f,\pi _{h}}^{C^{\ast
}}$ given in Example 5. Define the function%
\begin{equation*}
g_{0}:\left[ 0,1\right] \rightarrow 
%TCIMACRO{\U{211d} }%
%BeginExpansion
\mathbb{R}
%EndExpansion
,\ g_{0}\left( x\right) =\left\{ 
\begin{array}{c}
1,\text{ \ \ }x\in \left( 
%TCIMACRO{\U{211d} }%
%BeginExpansion
\mathbb{R}
%EndExpansion
\backslash 
%TCIMACRO{\U{211a} }%
%BeginExpansion
\mathbb{Q}
%EndExpansion
\right) \cap \left[ 0,1\right] \\ 
\infty ,\text{ \ \ }x\in 
%TCIMACRO{\U{211a} }%
%BeginExpansion
\mathbb{Q}
%EndExpansion
\cap \left[ 0,1\right] \text{ \ \ \ \ \ \ }%
\end{array}%
\right. .
\end{equation*}%
It is easy to show that $g_{0}\in L^{\infty }\left[ 0,1\right] .$ Let us
define the self-mapping $T:L^{\infty }\left[ 0,1\right] \rightarrow
L^{\infty }\left[ 0,1\right] $ as%
\begin{equation*}
Tg=\left\{ 
\begin{array}{c}
g,\text{ \ }g\in C_{f,\pi _{h}}^{C^{\ast }} \\ 
g_{0},g\notin C_{f,\pi _{h}}^{C^{\ast }}%
\end{array}%
\right. .
\end{equation*}%
Then, with a direct computation it can be seen that the self-mapping $T$
satisfies the conditions (2.2) and (2.3). That is to say that the circle $%
C_{f,\pi _{h}}^{C^{\ast }}$ is a fixed circle of $T.$
\end{example}

Now, we give an example of a self-mapping which satisfies the condition
(2.2) and does not satisfy the condition (2.3).

\begin{example}
Consider the $C^{\ast }-$algebra valued metric space $\left( 
%TCIMACRO{\U{211d} }%
%BeginExpansion
\mathbb{R}
%EndExpansion
,M_{2}\left( 
%TCIMACRO{\U{211d} }%
%BeginExpansion
\mathbb{R}
%EndExpansion
\right) ,d\right) $ given in Example 2 for $\alpha =3$ and the circle $%
C_{0,r_{1}}^{C^{\ast }}$ given in Example 6. Let us define the self-mapping $%
T:%
%TCIMACRO{\U{211d} }%
%BeginExpansion
\mathbb{R}
%EndExpansion
\rightarrow 
%TCIMACRO{\U{211d} }%
%BeginExpansion
\mathbb{R}
%EndExpansion
$ as%
\begin{equation*}
Tx=\left\{ 
\begin{array}{c}
\frac{1}{x},\text{\ \ }x\neq 0\text{ \ } \\ 
0,\text{ \ \ }x=0\text{ }%
\end{array}%
\right. .
\end{equation*}%
With a simple verification it can be shown that $T$ satisfies the condition
(2.2) and does not satisfy the condition (2.3). Notice that the circle $%
C_{0,r_{1}}^{C^{\ast }}$ is not a fixed circle of $T.$
\end{example}

Now, we present an example in order to show that Teorem 5 fail out, if the
condition (2.3) is satisfied but the condition (2.2) is not.

\begin{example}
Consider the $C^{\ast }-$algebra valued metric space $\left( X,%
%TCIMACRO{\U{211d} }%
%BeginExpansion
\mathbb{R}
%EndExpansion
^{2},d\right) $ given in Example 3 and the circle $C_{\left( \frac{-1}{2}%
,0\right) ,\left( 1,1\right) }^{C^{\ast }}$ given in Example 7. Let us
define the self-mapping $T:X\rightarrow X$ as%
\begin{equation*}
Tx=\left\{ 
\begin{array}{c}
\left( \frac{-3}{2},-2\right) ,\text{ \ \ \ }x=\left( \frac{-3}{2},-1\right)
\\ 
\left( \frac{-3}{2},2\right) ,\text{ \ \ \ }x=\left( \frac{-3}{2},1\right)
\\ 
\left( \frac{1}{2},-2\right) ,\text{ \ \ \ }x=\left( \frac{1}{2},-1\right)
\\ 
\left( \frac{1}{2},2\right) ,\text{ \ \ \ }x=\left( \frac{1}{2},1\right) \\ 
\text{ \ \ \ \ \ \ \ \ }\left( 1,5\right) ,~\text{ \ \ \ }x\notin C_{\left( 
\frac{-1}{2},0\right) ,\left( 1,1\right) }^{C^{\ast }}%
\end{array}%
\right. .
\end{equation*}%
Then, $T$ satisfies the condition (2.3) and does not satisfy the condition
(2.2). Notice that the circle $C_{\left( \frac{-1}{2},0\right) ,\left(
1,1\right) }^{C^{\ast }}$ is not a fixed circle of $T.$
\end{example}

We now state another existence theorem for fixed circles of a self-mapping
on a $C^{\ast }-$algebra valued metric space as follows:

\begin{theorem}
Let $\left( X,\mathbb{A},d\right) $ be a $C^{\ast }-$algebra valued metric
space, $C_{x_{0},r}^{C^{\ast }}$ be any circle on $X$ and the mapping $%
\varphi $ be as in (2.1). If $T$ is a self-mapping defined on $X$ satisfying
the conditions%
\begin{equation}
d\left( x,Tx\right) \preceq \varphi \left( x\right) +\varphi \left(
Tx\right) -2r
\end{equation}%
and%
\begin{equation}
d\left( Tx,x_{0}\right) \preceq r
\end{equation}%
for all $x\in C_{x_{0},r}^{C^{\ast }},$ then the circle $C_{x_{0},r}^{C^{%
\ast }}$ is a fixed circle of $T.$
\end{theorem}

\begin{proof}
Let $x$ be any point in the circle $C_{x_{0},r}^{C^{\ast }}.$ Then, using
the (2.4), (2.1), (2.5) and the definition of the relation $\preceq ,$ we
get 
\begin{eqnarray*}
d\left( x,Tx\right) &\preceq &\varphi \left( x\right) +\varphi \left(
Tx\right) -2r \\
&=&d\left( x,x_{0}\right) +d\left( Tx,x_{0}\right) -2r \\
&=&d\left( Tx,x_{0}\right) -r \\
&\preceq &r-r=0
\end{eqnarray*}%
and so $d\left( x,Tx\right) =0$ which implies that $Tx=x.$ Then, we deduce
that $C_{x_{0},r}^{C^{\ast }}$ is a fixed circle of $T.$
\end{proof}

\begin{remark}
The inequality (2.4) means that $Tx$ is not in the interior of the circle $%
C_{x_{0},r}^{C^{\ast }}$ for each $x\in C_{x_{0},r}^{C^{\ast }}.$ In the
same fashion, the inequality (2.5) means that $Tx$ is not in the exterior of
the circle $C_{x_{0},r}^{C^{\ast }}$ for each $x\in C_{x_{0},r}^{C^{\ast }}.$
These two results show that $T\left( C_{x_{0},r}^{C^{\ast }}\right) \subset
C_{x_{0},r}^{C^{\ast }}$ under the conditions (2.4) and (2.5).
\end{remark}

We continue with an example which satisfy the requirements of Theorem 6 as
follows:

\begin{example}
Consider the $C^{\ast }-$algebra valued metric space $\left( 
%TCIMACRO{\U{211d} }%
%BeginExpansion
\mathbb{R}
%EndExpansion
,M_{2}\left( 
%TCIMACRO{\U{211d} }%
%BeginExpansion
\mathbb{R}
%EndExpansion
\right) ,d\right) $ for $\alpha =2$ and the circle $C_{1,r_{2}}^{C^{\ast }}$
given in Example 9. Let us define the self-mapping $T:%
%TCIMACRO{\U{211d} }%
%BeginExpansion
\mathbb{R}
%EndExpansion
\rightarrow 
%TCIMACRO{\U{211d} }%
%BeginExpansion
\mathbb{R}
%EndExpansion
$ as%
\begin{equation*}
Tx=\left\{ 
\begin{array}{c}
x,\text{\ \ \ \ \ \ \ \ \ \ }x\in C_{1,r_{2}}^{C^{\ast }}\text{ \ } \\ 
1-\ln 3,\text{\ \ }x\notin C_{1,r_{2}}^{C^{\ast }}\text{ \ }%
\end{array}%
\right. .
\end{equation*}%
Clearly, $T$ satisfies the conditions (2.4) and (2.5), and we derive that
the circle $C_{1,r_{2}}^{C^{\ast }}$ is a fixed circle of $T.$
\end{example}

Now, we introduce an example of a self-mapping which satisfies the condition
(2.4) and does not satisfy the condition (2.5).

\begin{example}
Consider the $C^{\ast }-$algebra valued metric space $\left( X,%
%TCIMACRO{\U{211d} }%
%BeginExpansion
\mathbb{R}
%EndExpansion
^{2},d\right) $ given in Example 3, the circle $C_{\left( \frac{-1}{2}%
,0\right) ,\left( 1,1\right) }^{C^{\ast }}$ given in Example 7 and the
self-mapping $T:X\rightarrow X$ given in Example 12. Then, $T$ satisfies the
condition (2.4) and does not satisfy the condition (2.5). Notice that the
circle $C_{\left( \frac{-1}{2},0\right) ,\left( 1,1\right) }^{C^{\ast }}$ is
not a fixed circle of $T.$
\end{example}

Now, we furnish an example in order to evidence that Teorem 6 fail out, if
the condition (2.5) is satisfied but the condition (2.4) is not.

\begin{example}
Consider the $C^{\ast }-$algebra valued metric space $\left( 
%TCIMACRO{\U{211d} }%
%BeginExpansion
\mathbb{R}
%EndExpansion
,M_{2}\left( 
%TCIMACRO{\U{211d} }%
%BeginExpansion
\mathbb{R}
%EndExpansion
\right) ,d\right) $ given in Example 4 and the circle $C_{3,r}^{C^{\ast }}$
given in Example 8. Let us define the self-mapping $T:%
%TCIMACRO{\U{211d} }%
%BeginExpansion
\mathbb{R}
%EndExpansion
\rightarrow 
%TCIMACRO{\U{211d} }%
%BeginExpansion
\mathbb{R}
%EndExpansion
$ as%
\begin{equation*}
Tx=\left\{ 
\begin{array}{c}
3,\text{\ \ }x\in C_{3,r}^{C^{\ast }}\text{ \ } \\ 
0,\text{ \ \ }x\notin C_{3,r}^{C^{\ast }}%
\end{array}%
\right. .
\end{equation*}%
It is not hard to prove that $T$ satisfies the condition (2.5) and does not
satisfy the condition (2.4). Notice that the circle $C_{3,r}^{C^{\ast }}$ is
not a fixed circle of $T.$
\end{example}

The following theorem presents a new solution for fixed-circle problem
obtained by the help of the inequality (1.3) in Theorem 3.

\begin{theorem}
Let $\left( X,\mathbb{A},d\right) $ be a $C^{\ast }-$algebra valued metric
space, $C_{x_{0},r}^{C^{\ast }}$ be any circle on $X$ and the mapping $%
\varphi $ be as in (2.1). If $T$ is a self-mapping defined on $X$ satisfying
the conditions%
\begin{equation}
d\left( x,Tx\right) \preceq \varphi \left( x\right) -\varphi \left( Tx\right)
\end{equation}%
and%
\begin{equation}
r\preceq A^{\ast }d\left( x,Tx\right) A+d\left( Tx,x_{0}\right)
\end{equation}%
for all $x\in C_{x_{0},r}^{C^{\ast }}$ and some $A\in \mathbb{A}$ with $%
\left\Vert A\right\Vert <1,$ then the circle $C_{x_{0},r}^{C^{\ast }}$ is a
fixed circle of $T.$
\end{theorem}

\begin{proof}
Let $x$ be any point in the circle $C_{x_{0},r}^{C^{\ast }}.$ Let $x\neq Tx.$
Then, using the (2.6), (2.1) and (2.7) we get 
\begin{eqnarray*}
0 &\preceq &d\left( x,Tx\right) \preceq \varphi \left( x\right) -\varphi
\left( Tx\right) \\
&=&d\left( x,x_{0}\right) -d\left( Tx,x_{0}\right) \\
&=&r-d\left( Tx,x_{0}\right) \\
&\preceq &A^{\ast }d\left( x,Tx\right) A+d\left( Tx,x_{0}\right) -d\left(
Tx,x_{0}\right) \\
&=&A^{\ast }d\left( x,Tx\right) A
\end{eqnarray*}%
and so%
\begin{eqnarray*}
0 &\leq &\left\Vert d\left( x,Tx\right) \right\Vert \leq \left\Vert A^{\ast
}d\left( x,Tx\right) A\right\Vert \\
&\leq &\left\Vert A^{\ast }\right\Vert \left\Vert d\left( x,Tx\right)
\right\Vert \left\Vert A\right\Vert \\
&=&\left\Vert A\right\Vert ^{2}\left\Vert d\left( x,Tx\right) \right\Vert \\
&<&\left\Vert d\left( x,Tx\right) \right\Vert .
\end{eqnarray*}%
This yields a contradiction with our assumption. Therefore, we conclude that 
$x=Tx$ for all $x\in C_{x_{0},r}^{C^{\ast }}$ and hence $C_{x_{0},r}^{C^{%
\ast }}$ is a fixed circle of $T.$
\end{proof}

The following example support Theorem 7.

\begin{example}
Consider the $C^{\ast }-$algebra valued metric space $\left( 
%TCIMACRO{\U{211d} }%
%BeginExpansion
\mathbb{R}
%EndExpansion
,M_{2}\left( 
%TCIMACRO{\U{211d} }%
%BeginExpansion
\mathbb{R}
%EndExpansion
\right) ,d\right) $ for $\alpha =2$ and the circle $C_{1,r_{2}}^{C^{\ast }}$
given in Example 9 and the self-mapping $T:%
%TCIMACRO{\U{211d} }%
%BeginExpansion
\mathbb{R}
%EndExpansion
\rightarrow 
%TCIMACRO{\U{211d} }%
%BeginExpansion
\mathbb{R}
%EndExpansion
$ given in Example 13. Then, with a direct computation it can be seen that
the self-mapping $T$ satisfies the conditions (2.6) and (2.7) for $A=\left[ 
\begin{array}{cc}
\frac{1}{5} & 0 \\ 
0 & \frac{-3}{5}%
\end{array}%
\right] \in M_{2}\left( 
%TCIMACRO{\U{211d} }%
%BeginExpansion
\mathbb{R}
%EndExpansion
\right) $ with $\left\Vert A\right\Vert =\frac{3}{5}<1$. Note that the
circle $C_{1,r_{2}}^{C^{\ast }}$ is a fixed circle of $T.$
\end{example}

Now, we give an example of a self-mapping which satisfies the condition
(2.6) and does not satisfy the condition (2.7).

\begin{example}
Consider the $C^{\ast }-$algebra valued metric space $\left( 
%TCIMACRO{\U{211d} }%
%BeginExpansion
\mathbb{R}
%EndExpansion
,M_{2}\left( 
%TCIMACRO{\U{211d} }%
%BeginExpansion
\mathbb{R}
%EndExpansion
\right) ,d\right) $ given in Example 4, the circle $C_{3,r}^{C^{\ast }}$
given in Example 8 and the self-mapping $T:%
%TCIMACRO{\U{211d} }%
%BeginExpansion
\mathbb{R}
%EndExpansion
\rightarrow 
%TCIMACRO{\U{211d} }%
%BeginExpansion
\mathbb{R}
%EndExpansion
$ given in Example 15. It can be easily seen $T$ satisfies the condition
(2.6). But since there does not exist $A\in M_{2}\left( 
%TCIMACRO{\U{211d} }%
%BeginExpansion
\mathbb{R}
%EndExpansion
\right) $ with $\left\Vert A\right\Vert <1$ such that%
\begin{equation*}
\left[ 
\begin{array}{cc}
1 & 0 \\ 
0 & 1%
\end{array}%
\right] \preceq A^{\ast }\left[ 
\begin{array}{cc}
1 & 0 \\ 
0 & 1%
\end{array}%
\right] A,
\end{equation*}%
the condition (2.7) does not hold. Obviously, the circle $C_{3,r}^{C^{\ast
}} $ is not a fixed circle of $T.$
\end{example}

Now, we present an example in order to show that whether the Theorem 7 is
valid, if the condition (2.7) is satisfied but the condition (2.6) is not.

\begin{example}
Consider the $C^{\ast }-$algebra valued metric space $\left( 
%TCIMACRO{\U{211d} }%
%BeginExpansion
\mathbb{R}
%EndExpansion
,M_{2}\left( 
%TCIMACRO{\U{211d} }%
%BeginExpansion
\mathbb{R}
%EndExpansion
\right) ,d\right) $ given in Example 2 for $\alpha =3$ and the circle $%
C_{0,r_{1}}^{C^{\ast }}$ given in Example 6. Let us define the self-mapping $%
T:%
%TCIMACRO{\U{211d} }%
%BeginExpansion
\mathbb{R}
%EndExpansion
\rightarrow 
%TCIMACRO{\U{211d} }%
%BeginExpansion
\mathbb{R}
%EndExpansion
$ as%
\begin{equation*}
Tx=\left\{ 
\begin{array}{c}
2,\text{ \ \ \ }x\in C_{0,r_{1}}^{C^{\ast }} \\ 
-5,\text{ \ }x\notin C_{0,r_{1}}^{C^{\ast }}%
\end{array}%
\right. .
\end{equation*}%
It is not hard to show that $T$ satisfies the condition (2.7) for $A=\left[ 
\begin{array}{cc}
\frac{-1}{2} & 0 \\ 
0 & \frac{1}{2}%
\end{array}%
\right] \in M_{2}\left( 
%TCIMACRO{\U{211d} }%
%BeginExpansion
\mathbb{R}
%EndExpansion
\right) $ with $\left\Vert A\right\Vert =\frac{1}{2}<1$ and does not satisfy
the condition (2.6). Notice that the circle $C_{0,r_{1}}^{C^{\ast }}$ is not
a fixed circle of $T.$
\end{example}

We finish this section by give our last existence theorem for fixed circles
of a self-mapping on a $C^{\ast }-$algebra valued metric space as follows:

\begin{theorem}
Let $\left( X,\mathbb{A},d\right) $ be a $C^{\ast }-$algebra valued metric
space, $C_{x_{0},r}^{C^{\ast }}$ be any circle on $X$ and the mapping $%
\varphi $ be as in (2.1). $T$ is a self-mapping defined on $X$ satisfying
the condition%
\begin{equation}
A^{\ast }d\left( x,Tx\right) A\preceq \varphi \left( x\right) -\varphi
\left( Tx\right)
\end{equation}%
for all $x\in X$ where $A\in \mathbb{A}$ is an invertible element and $%
\left\Vert A^{-1}\right\Vert <1$ if and only if $T$ fixes the circle $%
C_{x_{0},r}^{C^{\ast }}$ and $T=I_{X}.$
\end{theorem}

\begin{proof}
Suppose that $T$ be a self-mapping defined on $X$ satisfying the condition
(2.8). Let $x$ be any point in $X.$ Let $x\neq Tx.$ Then, using the (2.8),
(2.1) and (iii) given in Definition 2 we get 
\begin{eqnarray*}
A^{\ast }d\left( x,Tx\right) A &\preceq &\varphi \left( x\right) -\varphi
\left( Tx\right) \\
&=&d\left( x,x_{0}\right) -d\left( Tx,x_{0}\right) \\
&\preceq &d\left( x,Tx\right) +d\left( Tx,x_{0}\right) -d\left(
Tx,x_{0}\right) \\
&=&d\left( x,Tx\right)
\end{eqnarray*}%
and so%
\begin{equation*}
d\left( x,Tx\right) \preceq \left( A^{\ast }\right) ^{-1}d\left( x,Tx\right)
A^{-1}=\left( A^{-1}\right) ^{\ast }d\left( x,Tx\right) A^{-1}.
\end{equation*}%
It follows that%
\begin{eqnarray*}
\left\Vert d\left( x,Tx\right) \right\Vert &\leq &\left\Vert \left(
A^{-1}\right) ^{\ast }d\left( x,Tx\right) A^{-1}\right\Vert \\
&=&\left\Vert \left( A^{-1}\right) ^{\ast }\right\Vert \left\Vert d\left(
x,Tx\right) \right\Vert \left\Vert A^{-1}\right\Vert \\
&=&\left\Vert A^{-1}\right\Vert ^{2}\left\Vert d\left( x,Tx\right)
\right\Vert \\
&<&\left\Vert d\left( x,Tx\right) \right\Vert .
\end{eqnarray*}%
But it is impossible. Hence, we can write $x=Tx$ for all $x\in X$ and $%
T=I_{X}.$

Conversely, suppose that $T$ fixes the circle $C_{x_{0},r}^{C^{\ast }}$ and $%
T=I_{X}.$ Then, since $Tx=x$ for all $x\in X,$ the condition (2.8) holds for
any invertible element $A\in \mathbb{A}$ with $\left\Vert A^{-1}\right\Vert
<1.$ This completes the proof.
\end{proof}

\begin{remark}
Theorem 8 says that if a self-mapping fixes a circle by satisfying the
conditions (2.2) and (2.3) (or the conditions (2.4) and (2.5)), but does not
satisfy the condition (2.8), then the self-mapping is different from the
identity map.
\end{remark}

\subsection{The uniqueness of fixed circles}

In this part, we discuss the uniqueness of fixed circles in the existence
theorems proved in subsection 2.1.

Before giving our uniqueness theorems, we emphasize that the fixed circles $%
C_{x_{0},r}^{C^{\ast }}$ in Theorem 5, Theorem 6 and Theorem 7 is not unique
with the following result.

\begin{proposition}
Let $\left( X,\mathbb{A},d\right) $ be a $C^{\ast }-$algebra valued metric
space and $C_{x_{1},r_{1}}^{C^{\ast }},C_{x_{2},r_{2}}^{C^{\ast
}},...,C_{x_{n},r_{n}}^{C^{\ast }}$ be any given circles. Then, there exists
at least one self-mapping $T$ of $X$ such that $T$ such that $T$ fixes all
the circles $C_{x_{1},r_{1}}^{C^{\ast }},$ $C_{x_{2},r_{2}}^{C^{\ast }},$ $%
...,C_{x_{n},r_{n}}^{C^{\ast }}.$
\end{proposition}

\begin{proof}
Let the self-mapping $T:X\rightarrow X$ be defined as%
\begin{equation*}
Tx=\left\{ 
\begin{array}{c}
x,x\in \overset{n}{\underset{i=1}{\cup }}C_{x_{i},r_{i}}^{C^{\ast }} \\ 
\alpha ,x\notin \overset{n}{\underset{i=1}{\cup }}C_{x_{i},r_{i}}^{C^{\ast }}%
\end{array}%
\right.
\end{equation*}%
where $\alpha \in X$ is a constant satisfying $d\left( \alpha ,x_{i}\right)
\neq r_{i}$ and the mapping $\varphi _{i}:X\rightarrow \left[ 0,\infty
\right) $ be defined as%
\begin{equation*}
\varphi _{i}\left( x\right) =d\left( x,x_{i}\right)
\end{equation*}%
for $i=1,2,...n.$ Then it is not hard to verify that the conditions (2.2)
and (2.3) in Theorem 5 are satisfied for the circles $C_{x_{1},r_{1}}^{C^{%
\ast }},C_{x_{2},r_{2}}^{C^{\ast }},...,C_{x_{n},r_{n}}^{C^{\ast }}.$
Therefore, all the circles $C_{x_{1},r_{1}}^{C^{\ast
}},C_{x_{2},r_{2}}^{C^{\ast }},...,C_{x_{n},r_{n}}^{C^{\ast }}$ are fixed
circles of the self-mapping $T$ by Theorem 5.
\end{proof}

Firstly, we focus on the uniqueness of fixed circles in Theorem 5 using the
inequality (1.1) in Theorem 1 in the following theorem.

\begin{theorem}
Let $\left( X,\mathbb{A},d\right) $ be a $C^{\ast }-$algebra valued metric
space, $C_{x_{0},r}^{C^{\ast }}$ be any circle on $X$ and the $T$ be a
self-mapping satisfying the conditions (2.2) and (2.3) given in Theorem 5.
If $T$ satisfies the contraction condition%
\begin{equation}
d\left( Tx,Ty\right) \preceq A^{\ast }d\left( x,y\right) A
\end{equation}%
for all $x\in C_{x_{0},r}^{C^{\ast }}$ , $y\in X-C_{x_{0},r}^{C^{\ast }}$
and some $A\in \mathbb{A}$ with $\left\Vert A\right\Vert <1,$ then the
circle $C_{x_{0},r}^{C^{\ast }}$ is unique fixed circle of $T.$
\end{theorem}

\begin{proof}
Assume that $C_{x_{1},\delta }^{C^{\ast }}$ is another fixed circle of $T.$
Let $a$ and $b$ be any points in $C_{x_{0},r}^{C^{\ast }}$ and $%
C_{x_{1},\delta }^{C^{\ast }},$ respectively. Then, we get by (2.9)%
\begin{equation*}
d\left( a,b\right) =d\left( Ta,Tb\right) \preceq A^{\ast }d\left( a,b\right)
A
\end{equation*}%
and so,%
\begin{equation*}
\left\Vert d\left( a,b\right) \right\Vert \leq \left\Vert A^{\ast }d\left(
a,b\right) A\right\Vert \leq \left\Vert A\right\Vert ^{2}\left\Vert d\left(
a,b\right) \right\Vert <\left\Vert d\left( a,b\right) \right\Vert .
\end{equation*}%
But this is impossible. Hence, the self-mapping $T$ fixes only circle $%
C_{x_{0},r}^{C^{\ast }}.$
\end{proof}

Now, we determine the uniqueness condition for the fixed circles in Theorem
6 using the condition (1.2) in Theorem 2.

\begin{theorem}
Let $\left( X,\mathbb{A},d\right) $ be a $C^{\ast }-$algebra valued metric
space, $C_{x_{0},r}^{C^{\ast }}$ be any circle on $X$ and the $T$ be a
self-mapping satisfying the conditions (2.4) and (2.5) given in Theorem 6.
If $T$ satisfies the contraction condition%
\begin{equation}
d\left( Tx,Ty\right) \preceq A\left( d\left( Tx,y\right) +d\left(
Ty,x\right) \right)
\end{equation}%
for all $x\in C_{x_{0},r}^{C^{\ast }}$ , $y\in X-C_{x_{0},r}^{C^{\ast }}$
and some $A\in \mathbb{A}_{+}^{\prime }$ with $\left\Vert A\right\Vert <%
\frac{1}{2},$ then the circle $C_{x_{0},r}^{C^{\ast }}$ is unique fixed
circle of $T.$
\end{theorem}

\begin{proof}
Assume that $C_{x_{1},\delta }^{C^{\ast }}$ is another fixed circle of $T.$
Let $a$ and $b$ be any points in $C_{x_{0},r}^{C^{\ast }}$ and $%
C_{x_{1},\delta }^{C^{\ast }},$ respectively. Then, we get by (2.10)%
\begin{equation*}
d\left( a,b\right) =d\left( Ta,Tb\right) \preceq A\left( d\left( Ta,b\right)
+d\left( Tb,a\right) \right) ,
\end{equation*}%
so that%
\begin{eqnarray*}
\left\Vert d\left( a,b\right) \right\Vert &\leq &\left\Vert A\left( d\left(
Ta,b\right) +d\left( Tb,a\right) \right) \right\Vert \\
&\leq &\left\Vert A\right\Vert \left\Vert d\left( a,b\right) +d\left(
b,a\right) \right\Vert \\
&=&2\left\Vert A\right\Vert \left\Vert d\left( a,b\right) \right\Vert \\
&<&\left\Vert d\left( a,b\right) \right\Vert
\end{eqnarray*}%
which is a contradiction which means that $a=b.$ This shows that the
self-mapping $T$ fixes only circle $C_{x_{0},r}^{C^{\ast }}.$
\end{proof}

Finally, we state our last uniqueness theorem for the fixed circles in
Theorem 7 using the condition (1.4) in Theorem 4.

\begin{theorem}
Let $\left( X,\mathbb{A},d\right) $ be a $C^{\ast }-$algebra valued metric
space, $C_{x_{0},r}^{C^{\ast }}$ be any circle on $X$ and the $T$ be a
self-mapping satisfying the conditions (2.6) and (2.7) given in Theorem 7.
If $T$ satisfies the contraction condition that there exists $u\left(
x,y\right) \in \left\{ d\left( x,y\right) ,d\left( x,Tx\right) ,d\left(
y,Ty\right) ,d\left( x,Ty\right) ,d\left( y,Tx\right) \right\} $ such that%
\begin{equation}
d\left( Tx,Ty\right) \preceq A^{\ast }u\left( x,y\right) A
\end{equation}%
for all $x\in C_{x_{0},r}^{C^{\ast }}$ , $y\in X-C_{x_{0},r}^{C^{\ast }}$
and some $A\in \mathbb{A}$ with $\left\Vert A\right\Vert <1,$ then the
circle $C_{x_{0},r}^{C^{\ast }}$ is unique fixed circle of $T.$
\end{theorem}

\begin{proof}
Assume that $C_{x_{1},\delta }^{C^{\ast }}$ is another fixed circle of $T.$
Let $a$ and $b$ be any points in $C_{x_{0},r}^{C^{\ast }}$ and $%
C_{x_{1},\delta }^{C^{\ast }},$ respectively. Then, we get by (2.11)%
\begin{equation*}
d\left( a,b\right) =d\left( Ta,Tb\right) \preceq A^{\ast }u\left( a,b\right)
A,
\end{equation*}%
so that%
\begin{eqnarray*}
\left\Vert d\left( a,b\right) \right\Vert &\leq &\left\Vert A^{\ast }u\left(
a,b\right) A\right\Vert \\
&\leq &\left\Vert A\right\Vert ^{2}\left\Vert u\left( a,b\right) \right\Vert
\\
&<&\left\Vert u\left( a,b\right) \right\Vert \\
&\leq &\max \left\{ \left\Vert d\left( a,b\right) \right\Vert ,\left\Vert
d\left( a,Tb\right) \right\Vert ,\left\Vert d\left( b,Tb\right) \right\Vert
,\left\Vert d\left( a,Tb\right) \right\Vert ,\left\Vert d\left( b,Ta\right)
\right\Vert \right\} \\
&=&\left\Vert A\right\Vert ^{2}\max \left\{ \left\Vert d\left( a,b\right)
\right\Vert ,0\right\} \\
&=&\left\Vert d\left( a,b\right) \right\Vert
\end{eqnarray*}%
which is a contradiction. Hence $a=b.$ This implies that the self-mapping $T$
fixes only circle $C_{x_{0},r}^{C^{\ast }}.$
\end{proof}

\begin{remark}
1) The uniqueness result in Theorem 5 can be also stated using the
contraction conditions (2.10) given in Theorem 10 or (2.11) given in Theorem
11 instead of the contraction condition (2.9).
\end{remark}

\qquad \textit{2) The uniqueness result in Theorem 6 can be also stated
using the contraction conditions (2.9) given in Theorem 9 or (2.11) given in
Theorem 11 instead of the contraction condition (2.10).}

\qquad \textit{3) The uniqueness result in Theorem 7 can be also stated
using the contraction conditions (2.9) given in Theorem 9 or (2.10) given in
Theorem 10 instead of the contraction condition (2.11).}

\section{Conclusion and Future Works}

In this article, we discuss the existence and uniqueness of the fixed-circle
for self-mappings satisfying some special conditions on $C^{\ast }-$algebra
valued metric spaces. We also furnish some examples to show effectiveness of
our theoretical results. Using similar approaches and new contractive
conditions, it can be studied new fixed-circle results on $C^{\ast }-$%
algebra valued metric spaces. Since with this article we develop a new and
different perspective instead of the classical fixed point theory in $%
C^{\ast }-$algebra valued metric spaces, we hope that our results will
support researchers for future works and applications to other related areas.

\end{document}